\documentclass[a4paper,11pt,twoside,DIV=calc]{scrartcl}	

\input xy
\xyoption{all}
\usepackage{lscape}
\usepackage[latin1]{inputenc}
\usepackage[cmtip,arrow]{xy}
\usepackage{amsmath}
\usepackage{amssymb,amscd,amsthm,mathrsfs,yfonts,manfnt,pb-diagram,pb-xy}
\usepackage[nice]{nicefrac}		
\usepackage[T1]{fontenc}
\usepackage{lmodern} 					
\usepackage{tocenter}					
\ToCenter[h]{\textwidth}{\textheight}	
\usepackage{ulem}							
\usepackage{color}

\renewcommand{\emph}{\textit}		

\usepackage{ifthen}
\newboolean{comm}
\setboolean{comm}{false}
\newboolean{sol}
\setboolean{sol}{false}
\newboolean{notes}
\setboolean{notes}{false}			
\CompileMatrices  						

\newcommand{\com}{\ifthenelse{\boolean{comm}}}
\newcommand{\sol}{\ifthenelse{\boolean{sol}}}
\newcommand{\note}{\ifthenelse{\boolean{notes}}}

\newtheorem{Def}{Definition}
\newtheorem{Prop}[Def]{Proposition}

\newtheorem{Th}[Def]{Theorem}

\newtheorem{Lem}[Def]{Lemma}
\newtheorem{Cor}[Def]{Corollary}

\newtheorem{PD}[Def]{Proposition/Definition}
\theoremstyle{definition}
\newtheorem{Rem}[Def]{Remark}

\newcommand{\mR}{\ensuremath{\mathbb{R}}}								
\newcommand{\mC}{\ensuremath{\mathbb{C}}}					

\newcommand{\mZ}{\ensuremath{\mathbb{Z}}}

\newcommand{\mB}{\ensuremath{\mathbb{B}}}

\newcommand{\mc}{\mathcal}								

\newcommand{\linspan}{\text{span}}

\DeclareMathOperator{\Image}{Im}		

\renewcommand{\Im}{\Image}							
\newcommand{\Scp}{\langle\;\cdot\;|\;\cdot\;\rangle}		

\newcommand{\ee}{{\mbox{$\varepsilon$}}}

\renewcommand{\phi}{\varphi}

\newcommand{\sph}{\mathbb{S}}

\newcommand{\fn}{\mathcal{C}}	   				  

\newcommand{\mb}{\begin{pmatrix}}					
\newcommand{\me}{\end{pmatrix}}						
\newcommand{\hsp}{\mc{H}}									
\newcommand{\bo}{\mathbb{B}(\mc{H})}			

\DeclareMathOperator{\ind}{ind}						
\newcommand{\Dom}{\mathcal{D}}						
\newcommand{\Cliff}{\text{Cliff}}					
\newcommand{\Cliffc}{\mC l}

\newcommand{\stoe}{\ensuremath{\mathcal{T}^\infty}}		
\newcommand{\sco}{\ensuremath{\mathcal{K}}}						
\newcommand{\ssu}{\ensuremath{\mathcal{S}}}						
\newcommand{\sCo}{\ensuremath{\mathcal{C}}}						
\newcommand{\potimes}{\otimes_{\pi}}									

\newcommand{\mcA}{\mc{A}}															
\newcommand{\mcB}{\mc{B}}
\newcommand{\mcC}{\mc{C}}

\newcommand{\mcS}{\mc{S}}

\DeclareMathOperator{\Ker}{Ker}						
\DeclareMathOperator{\Coker}{Coker}
\DeclareMathOperator{\Hom}{Hom}						
\DeclareMathOperator{\id}{id}							

\usepackage{tocenter}				
\ToCenter[h]{\textwidth}{\textheight}	

\setkomafont{section}{\bfseries\Large}		

\usepackage{scrpage2}      	
\ofoot{}					
\ifoot{}
\rehead{\pagemark}			
\cehead{MARTIN GRENSING}	
\rohead{\pagemark}			
\cohead{The Thom isomorphism in bivariant $k$-theory}

\begin{document}

\title{\Large{{THE THOM ISOMORPHISM IN BIVARIANT K-THEORY}}}
\author{\small{MARTIN GRENSING}}
\date{\small{\today}}
\maketitle
\thispagestyle{empty}
\begin{abstract} \small{We give a simple proof of the smooth Thom isomorphism for complex bundles for the bivariant $K$-theories on locally convex algebras considered by Cuntz. We also prove the Thom isomorphism in Kasparov's $KK$-theory in a form stated without proof in the conspectus (\cite{MR1388299}).}
\end{abstract}
 \section{Introduction}
The classical proof of the Thom isomorphism  in $K$-theory for a complex bundle $E\to X$ over a compact space $X$ is based on the Thom elements $\lambda_E$ (see \cite{MR0224083}). These Thom elements are certain classes in the reduced $K$-theory of Thom space $E^+$ of the bundle. Using an orthogonal of the bundle, one may then deduce the Thom isomorphism from Bott periodicity.

We here define the analogous smooth bivariant Thom classes in Cuntz's $kk$ (\cite{MR1456322}). These smooth bivariant Thom classes have  certain multiplicativity properties which, combined  with elementary formal properties of $kk$, yield a very simple proof of the Thom isomorphism for $kk$. As $kk$ is universal for diffotopy invariant, half exact and stable functors, this yields the Thom isomorphism for any such functor. In particular, using the bivariant Chern-Connes character from \cite{MR1456322}, we obtain the Thom isomorphism in bi- and monovariant periodic cyclic cohomology. The analogous result for split exact functors (\cite{UC}) is more difficult to obtain, and requires entirely different techniques.

We then proceed to give a proof of Bott periodicity and  a version of the Thom isomorphism that was stated without proof in \cite{MR1388299}. This is essentially based on the techniques of \cite{KaspOp}, but we use instead the connection formalism of Connes and Skandalis \cite{MR775126}, and deduce Bott periodicity basically  from the theory of the harmonic oscillator. We base our proof on the functoriality of $KK$ for principal bundles.

We note that the Thom elements $\lambda_E$, and hence the smooth bivariant Thom classes in $kk$, may also be obtained as restrictions of the fibred versions  of the Bott element used below in the proof of the Thom isomorphism in $KK$ by replacing the operator and applying a Morita equivalence. 

I would like to thank G. Skandalis and J. Cuntz for helpful discussions and remarks.
\section{The Thom Isomorphism in $kk$}\label{kkthom}
\subsection{$kk$-theory}
We recall the definition of $kk$-theory as given by Cuntz in \cite{MR2240217}, to which we refer for details. 
\begin{Def}\label{semisplit} An exact sequence 
\[\xymatrix{0\ar[r]&{\mcC}\ar[r]&{\mcB}\ar[r]&\mcA\ar[r]&0}\]
of locally convex algebras will be called semisplit if it is split as a sequence of topological vector spaces. It will be called a split exact sequence, if it is split in the category of locally convex algebras. 
\end{Def}

The essential ingredients for the theory $kk$ are as follows:
\begin{itemize}
\item The smooth compacts $\sco$, which are the algebra of matrices with rapidly decreasing entries,\index{$\sco$}
\item the smooth Toeplitz algebra,\index{$\stoe$} which is as a vector space $\sco\oplus\fn^{\infty}(\sph)$, and a locally convex algebra with the multiplication it inherits from the $C^*$-Toeplitz algebra,
\item the resulting Toeplitz extension 
\[\xymatrix{0\ar[r]&\sco\ar[r]&\stoe\ar[r]&\fn^\infty(\sph)\ar[r]&0}\]
\item the cone extension 
\[\xymatrix{0\ar[r]&\ssu\ar[r]&\sCo\ar[r]&\mC\ar[r]&0}\]
\item the universal semisplit  extension (cf. Definition \ref{semisplit}) of a locally convex algebra $A$\index{$TA$}\index{$JA$}
\[\xymatrix{0\ar[r]&JA\ar[r]&TA\ar[r]&A\ar[r]&0}\]
 where $TA$ is the locally convex tensor algebra of the vector space $A$, and the quotient map the map induced by universality from $\id_A$, $JA$ the kernel of this map. It also has a continuous linear split $A\to TA$ given by the inclusion. Using the universality of $TA$, one shows that this is in fact the initial object in the category of semisplit extensions of $A$. More generally:
\end{itemize}
\begin{Prop}
For any semisplit extension of $A'$ and morphism $\phi:A\to A'$ (this is called a semisplit morphism extension) there are maps $\sigma$, $\psi$ such that 
\[\xymatrix{0\ar[rr]&&				JA\ar[rr]\ar[d]_{\sigma}&&	TA\ar[rr]\ar[d]^{\psi}&& A	\ar[rr]\ar[d]^{\phi}&&						0\\
   0\ar[rr]&&				B\ar[rr]&&							 	C\ar[rr]&&		A'\ar[rr]\ar@/_.3cm/[ll]^s&&			    			0	}\]
	commutes.
	\end{Prop}
The map $\psi$ is the map induced from $s\circ\phi$ on $TA$ by universality, and it restricts by commutativity to $\sigma:JA\to B$.
	
\begin{Def} The map $\sigma$ of the last definition is called the classifying map of the morphism extension.
\end{Def}
The (diffotopy class of a) classifying map of such a morphism extension does not depend on the split (join two different splits by the obvious linear path). More generally, we have classifying maps $J^nA\to B$ for an $n$-step extension.  We define $T^n A$ and $J^nA$ by iteration of the functors $J$ and $T$ (this is independent of the order up to diffotopy).

\begin{Def} The groups $kk^{alg}_m(A,B)$\index{$kk^{alg}_m(A,B)$} for locally convex algebras $A$ and $B$ are now defined as the direct limit
$$\lim_{\to} \left(\Hom(J^{m+n} A,\sco\otimes_\pi \ssu^n B)/\sim\right),$$
and \index{$kk^{\mc{L}^p}(A,B)$}$kk^{\mc{L}^p}_m(A,B):=kk^{alg}_m(A,B\otimes_\pi \mc{L}^p)$. The class of a homomorphism $\phi$ will be denoted $\langle \phi\rangle$.
\end{Def}
Here $\sim$ denotes diffotopy, and the connecting maps $\Lambda_n$\index{$\Lambda_n$} are given by sending  $\phi:J^{n+m} A\to \sco\otimes_\pi\ssu^n B$ to the classifying map of
\[\xymatrix{&&& J^{n+m} A	\ar[d]^{\phi}\\
   0\ar[r]&			\sco\otimes_\pi \ssu^{n+1} B\ar[r]&					 	\sco\otimes_\pi \sCo\ssu^n B \ar[r]&		\sco\otimes_\pi \ssu^n B\ar[r]&			    			0 .}\]
To simplify notation, we write $kk$ for $kk^{alg}$; further we will suppress the tensor product by $\sco$ whenever this does not essentially change the arguments.  

We note that $kk^{\mc{L}^p}$ does not depend on $p$ (as long as $p$ is large enough). The further properties of the theory may be nicely developed using triangulated categories and  postponing the stabilisation by $\sco$ or $\mc{L}^p$ to the end. We refer to \cite{CuntzMeyer} for further details on this.
\subsection{The outer product in $kk$}

Denote by $s$ the canonical map $J^m(A\otimes_\pi B)\to J^m(A)\otimes_\pi B$. We will define a multiplicative transformation \index{$\tau_D$}
$$\tau_D:kk_*^{alg}(A,B)\to kk_*^{alg}(A\otimes D,B\otimes D)$$ 
as follows: For $\phi:J^m A\to \ssu^n B$, set
$$\tau_D\langle \phi\rangle:=\langle J^m(A\otimes_\pi D)\overset{s}{\rightarrow}J^m A\otimes_\pi D\overset{\phi\otimes_\pi\id_D}{\longrightarrow}\ssu^n B\otimes_\pi D\rangle.$$
 To show that $\tau_D$ is well defined, it suffices to show $\tau_D\langle \phi\rangle=\tau_D\langle\Lambda \phi\rangle$, which will follow from
 \begin{align*}&\langle\Lambda(J^m(A\otimes_\pi D)\stackrel{s}{\rightarrow}J^m A\otimes_\pi D \overset{\phi\otimes id_D}{\longrightarrow}
 \ssu^n B\otimes_\pi D\rangle\\
=&\langle J^{m+1}(A\otimes_\pi D)\overset{s}{\rightarrow}
 J^{m+1} A\otimes_\pi D\overset{\Lambda\phi\otimes \id_D}{\longrightarrow}\ssu^{n+1}B\otimes_\pi D\rangle.
\end{align*}
 The equality of these maps is obtained by the usual arguments, using the defining diagrams for both maps.
\ifthenelse{\boolean{notes}}{\marginpar{proof}}{}
The following easy proposition determines the action of $\tau_D$ on elements of $kk^{alg}$ given by extensions:
\begin{Prop}
If $\langle\phi\rangle$ is the class in $kk_1(A,B)$ of the classifying map of a semisplit extension
$$\xymatrix{ 0\ar[r]&B\ar[r]&C\ar[r]&A\ar[r]&0}$$
then $\tau_D\langle\phi\rangle$  is the class of the classifying map of
$$\xymatrix{ 0\ar[r]&B\otimes_\pi D\ar[r]&C\otimes_\pi D\ar[r]&A\otimes_\pi D\ar[r]&0}.$$
An analogous formula holds for extensions of higher length.
\end{Prop}
 \begin{proof}
 This follows by uniqueness of classifying maps and commutativity of
 $$\xymatrix{ 
	 0\ar[rr]&&	J(A\otimes D)\ar[rr]\ar[d]^s&&	T(A\otimes D)\ar[rr]\ar[d]&&	A\otimes D\ar[rr]\ar[d]_{\id}&&			0\\
	 0\ar[rr]&&				JA\otimes D\ar[rr]\ar[d]^{\phi\otimes \id_D}&&	TA\otimes D\ar[rr]\ar[d]&&	A\otimes D\ar[rr]\ar[d]_{\id}&&						0\\
   0\ar[rr]&&				B\otimes D\ar[rr]&&							 	C\otimes D\ar[rr]&&		A\otimes D\ar[rr]&&			    			0	}$$
 where the outer edge is defining the classifying map of the extension tensored by $D$, and is therefore diffotopic to $(\phi\otimes \id_D)\circ s$.
\end{proof}

\begin{Def}
We define the cup product by \index{$\cup$}
$$kk_i(A,B)\times kk_j( A^\prime,B^\prime) \to kk_{i+j}(A\otimes A^\prime,  B\otimes B^\prime)$$
$$(\alpha,\beta) \mapsto \alpha\cup\beta:=\tau_{A^\prime}(\alpha)\cap\tau_B(\beta)$$
\end{Def}

As an easy consequence of the above proposition, we may describe the cup product of the classes of two extensions
$$\xymatrix{ 0\ar[r]&B\ar[r]&C\ar[r]&A\ar[r]&0}, \xymatrix{ 0\ar[r]&B^\prime\ar[r]&C^\prime\ar[r]&A^\prime\ar[r]&0}$$
 by the class of the length two extension
 $$\xymatrix{ 0\ar[r]&B\otimes_\pi B^\prime \ar[r]&C\otimes_\pi B^\prime\ar[r]&A\otimes_\pi C^\prime \ar[r]& A\otimes_\pi A^\prime\ar[r]&0}$$
 \ifthenelse{\boolean{notes}}{\marginpar{In Karoubi on Bott periodicity: The $\cup$ is graded commutative!}}{}
 We give two easy properties of the cup product:
 \begin{Prop}
The cup product is graded commutative and the following formula describes the relation between the cup and cap product
 $$(\alpha\cup\beta)\cap(\gamma\cup\delta)=(\alpha\cap\gamma)\cup(\beta\cap \delta),$$
 where $\alpha\in kk(A,B)$, $\beta\in kk(A^\prime,B^\prime)$, $\gamma\in kk(B,C)$ and $\delta\in kk(B^\prime, C^\prime)$.
 \end{Prop}
\begin{proof}
Represent $\alpha$ by $\phi:J^m A\to \ssu^{\underline{m}} B$ and $\beta$ by $\psi:J^nA^\prime\to \ssu^{\underline{n}} B^\prime$. Then the commutativity of $\cup$ up to sign follows by the usual arguments from commutativity of 
\[\scalebox{0.85}{\xymatrix{ 
	 0\ar[rr]&&	J^{m+n}(A\otimes A^\prime)\ar[rr]\ar[d]&&				J^m(A\otimes J^n A^\prime)\ar[rr]^{J^m(\id_A\otimes \psi)}\ar[d]&&	J^m(A\otimes \ssu^{\underline{n}}B^\prime)\ar[rr]\ar[d]&&						0\\
	 0\ar[rr]&&		J^n(J^mA\otimes A^\prime)\ar[rr]\ar[d]^{J^n(\phi\otimes A^\prime)}&&	J^mA\otimes J^n A^\prime\ar[rr]^{J^m(\id_A)\otimes \psi}\ar[d]^{\phi\otimes J^n A^\prime}&&							  	J^m(A)\otimes \ssu^{\underline{n}}B\ar[rr]\ar[d]^{\phi\otimes \ssu^{\underline{n}}B}&&			0\\
   0\ar[rr]&&								J^n(\ssu^{\underline{m}}B\otimes A^\prime)\ar[rr]&&							 \ssu^{\underline{m}}B\otimes J^n A^\prime\ar[rr]_{\ssu^{\underline{m}}B\otimes \psi}&&															  \ssu^{\underline{m}}B\otimes \ssu^{\underline{n}} B\ar[rr]&&	0	}}\]
To prove the formula given above, we calculate
\begin{align*}
(\alpha\cup\beta)\cap(\gamma\cup\delta)=&\tau_{A^\prime}(\alpha)\cap\tau_B(\beta)\cap \tau_{B^\prime}(\gamma)\cap\tau_C(\delta)\\
																			=&\tau_{A^\prime}(\alpha)\cap\tau_{A^\prime}(\gamma)\cap\tau_C(\beta)\cap\tau_C(\delta)\\
																			\stackrel{*}{=}&\tau_{A^\prime}(\alpha\cap\gamma)\cap\tau_C(\beta\cap\delta)\\
																			=&(\alpha\cap\gamma)\cup(\beta\cap\delta),
\end{align*}
where we use multiplicativity of $\tau$ at $*$.
\end{proof}

In particular, the formula above shows that the cup product $\alpha\cup\beta$ of invertible elements is invertible  with inverse $\pm(\alpha^{-1}\cup\beta^{-1})$ (invertibility with respect to the cap product).

\subsection{The Thom isomorphism in $kk_m$}
This proof uses $kk_m$, the version of $kk$ theory developed by Cuntz in \cite{MR1456322}, which is a bivariant $K$-theory on the category of $m$-algebras.  It nicely reduces the computations of products almost entirely to $K$-theory, where one may use the classical results. We denote throughout this section by $\Sigma$\index{$\Sigma$} the reduced suspension, i.e., the sphere with base point the north pole, and by $H$ \index{$H$}the tautological line bundle on the two-sphere.\note{Atiyah definition before 2.2.1}{} We first formulate the ingredients for this proof:
\begin{enumerate}
\item The theory $kk_m$ \index{$kk_m$}as developed in \cite{MR1456322}, and the fact that it has Bott periodicity (\cite{MR1456322}, Satz 5.4)
\item Phillips K-theory for Frechet-algebras (for this paragraph, à Frechet algebra is complex, locally multiplicatively convex algebra which is a Frechet space) from
 \cite{MR1082838}
\item\label{comparisonK} The comparison between $kk_m(\mC,\cdot)$ and Phillips $K$-theory from \cite{MR1456322}, Theorem 7.4
\item The multiplicative elements $\lambda_E\in\tilde K^0(X^E,\infty)=K_0(\fn_0(E)\tilde{})$ for a bundle $E\rightarrow X$ ($X^E$ denotes the Thom space of $V$)\index{$X^E$}
\end{enumerate}
Let us first comment on the relation between the first three items: The  isomorphism in \ref{comparisonK}   is constructed in  \cite{MR1456322}, Chapter 7, and it is easily seen to be compatible with outer products.  One thus sees that the class $\langle H\rangle-1$, i.e., the generator of $\tilde K^0(\Sigma^2)(\approx K_0(S^2))$ in topological $K$-theory, is mapped to a  generator in $kk_m(\mC,\ssu^2)$, by using that the $K$-theory of the smooth subalgebra $\ssu^2$ coincides with the reduced topological $K$-theory of $\Sigma^2$. Further, as everything commutes with outer products, the generator of $\tilde K(\Sigma^{2n})$  yields (up to sign) the canonical generator $y_{2n}$ of the group $kk(\mC,\ssu^{2n})$  for every $n$. We denote $x_{2n}$ the inverse of $y_{2n}$.

Our goal is to introduce bivariant Thom elements. Throughout this section, all bundles $E\twoheadrightarrow X$ are taken to be smooth complex bundles over a  manifold. We denote by $\mc{S}(E)$\index{$\mcS(E)$} the smooth functions on the total space of the bundle that are Schwartz functions along the fibre (compare also \cite{ThesePaulo}). We define\index{$k_m(A)$} $k_m(A):=kk_m(\mC,A)$ for every locally convex algebra.

We recall from \cite{MR0488029}, chapter 4, and Atiyah \cite{MR0224083}, p. 100 that the Thom-element for a complex bundle over a compact space is a certain class  $\lambda_E\in K(\fn_0(E))$. Because algebraic $K$-theory is invariant under passage to smooth subalgebras, and $k_m$ coincides with Phillips $K$-theory for Fréchet algebras, we may view $\lambda_E$ as elements in $k_m(\mcS(E))$ by choosing a smooth representative of $\lambda_E$.  The properties of the elements $\lambda_E$ in topological $K$-theory then yield the following proposition.
\begin{Prop} For every smooth complex bundle $\pi:E\twoheadrightarrow X$ over a compact manifold, there is an element $\lambda_E\in k_m(\mcS(E))$ \index{$\lambda_E$}such that the following properties hold
\begin{enumerate}
\item If $E $ is the trivial $n$-dimensional complex bundle, then $\lambda_E=[1_{\mcS(X)}]\cup y_{2n}\in k_m(\mcS(X)\potimes \mcS(\mC^n))$.
\item $\lambda_E\cup\lambda_{E'}=\lambda_{E\times E'}$, where $E'\twoheadrightarrow Y$ is a bundle over a compact manifold $Y$ and $E\times E'$ denotes the product bundle over $X\times Y$.
\end{enumerate}
\end{Prop}
Using these elements, we  define bivariant Thom elements:
\begin{Def} For every two smooth complex bundles $\pi:E\twoheadrightarrow X$,  $\pi':E'\twoheadrightarrow X$   over a compact manifold $X$ we denote by $\phi_{E,E'}:\mcS(E\times E')\to \mcS(E\oplus E')$ the map induced by the inclusion $E\oplus E'\hookrightarrow E\times E'$,\index{$\phi_{E,E'}$} and set \index{$t_{E,E'}$}
$$t_{E,E'}:={\phi_{E,E'}}_*(\tau_{\mcS(E)}(\lambda_{E'}))\in kk_m(\mcS(E),\mcS(E\oplus E')).$$
\end{Def}
The bivariant Thom elements have the following properties:
\begin{Lem} Let $X$ be a compact manifold.
\begin{enumerate} 
\item Let $\pi':E'\twoheadrightarrow X$ be the complex $n$-dimensional trivial bundle. Then 
$$t_{E,E'}=\tau_{\mcS(E)}(y_{2n})\in kk_m(\mcS(E),\mcS(E)\potimes\mcS(\mR^n)).$$
\item Let $\pi:E\twoheadrightarrow X$, $\pi':E'\twoheadrightarrow X$ be two bundles. Then
$$t_{E,E'}\cap t_{E\oplus E',E''}=t_{E,E'\oplus E''}.$$
\end{enumerate}
\end{Lem}
\begin{proof} 
The first statement is straightforward. We abbreviate $\tau_{\mcS(E)}$ by $\tau_E$  for a given bundle. Then we have 
\begin{equation}\label{phi} {\phi_{E,E'}}^*(\tau_{E\oplus E'}(\lambda_{E''}))=(\phi_{E,E'}\otimes \id_{\mcS(E'')})_*(\tau_{E\times E'}(\lambda_{E''}))
\end{equation}
and 
\begin{equation}\label{trivialembed} \theta:=\phi_{E\oplus E',E''}\circ (\phi_{E,E'}\otimes \id_{\mcS(E'')})=\phi_{E,E'\oplus E''}\circ(\id_{\mcS(E)}\otimes \phi_{E',E''}).
\end{equation}
Thus we get
\begin{align*}
t_{E,E'}\cap t_{E\oplus E',E''}=&{\phi_{E,E'}}_*(\tau_E(\lambda_{E'}))\cap {\phi_{E\oplus E',E''}}_*(\tau_{E\oplus E'}(\lambda_{E''}))\\
=&{\phi_{E\oplus E',E''}}_*(\tau_E(\lambda_{E'})\cap {\phi_{E,E'}}^*(\tau_{E\oplus E'}({\lambda_{E''}})))\\ 
=&\theta_*(\tau_E(\lambda_{E'})\cap (\tau_{E\times E'}({\lambda_{E''}})))\\
=&\theta_*(\tau_E(\lambda_{E'}\cup \lambda_{E''})))\\
=&\theta_*(\tau_E(\lambda_{E'\times E''})))\\
=&{\phi_{E,E'\oplus E''}}_*(\tau_E(\lambda_{E'\oplus E''}))\\
=&t_{E,E'\oplus E''}
\end{align*} \end{proof}
\begin{Th}[The smooth Thom isomorphism] Let $X$ be a closed manifold. There is an invertible element in $kk_m(\mcS(X),\mcS(E))$ for every complex bundle $E\twoheadrightarrow X$.
\end{Th}
\begin{proof} Let $E^\perp$ be a bundle such that $E\oplus E^\perp$ is trivial. We have by the above Lemma:
$$t_{X,E}\cap t_{E,E^{\perp}}=t_{X,E\oplus E^\perp}=\tau_{\mcS(X)}(y_{n}),$$
for some $n$, thus $t_{X,E}$ is right invertible. So is $t_{E,E^\perp}$ by
$$t_{E,E^\perp}\cap t_{E\oplus E^\perp,E}=t_{E, E^\perp\oplus E}=\tau_{\mcS(E)}(y_m).$$
Thus $t_{X,E}$ is right invertible with invertible right inverse, thus invertible.\end{proof}

\section{Bott periodicity and the Thom isomorphism in $KK$}\label{cstarbott}

\subsection{Some preliminaries concerning Clifford algebras}\label{prepcliff}
\newcommand{\ees}{{\ee^*}}
We fix throughout a real vector space $V$ of dimension $n$ with a positive definite scalar product $\Scp$, and write $\Cliff (V,\Scp)$ for the Clifford algebra of $V$, where we use the convention $v^2=+\langle v|v\rangle$. We denote by $\Cliffc(V)$ the complexification of $\Cliff(V)$, and abbreviate $\Cliffc (\mR^n):=\mC_n$\index{$\mC_n$}. See \cite{LawsonMichelsohn} for further details on Clifford algebras. For $v\in V$, we denote by $\ee(v)\in\mB(\Lambda_\mC^*V)$\index{$\ee$} the operator of exterior multiplication with $v$. We view $\Lambda_\mC^* V$ as a Hilbert space with the canonical scalar product. Recall that the Clifford algebra is functorial with respect to linear maps preserving the scalar product, and $\Cliff(V)^{(i)}$, $i=0,1$, denotes the $(-1)^{i}$-eigenspaces of the map $\gamma$ induced by $v\mapsto -v$; in particular, the exterior algebra is graded by $\gamma$.

\renewcommand{\ees}{\bar\ee}
Recall that the CAR algebra is the algebra generated by $\ee_k$, $\ees_l$ subject to the relations
$$\{\ee_k,\ees_l\}=\delta_{kl},\;\; \{\ee_k,\ee_l\}=0=\{\ees_k,\ees_l\},$$
where $\{a,b\}=ab+ba$ denotes the anticommutator. 

\begin{Prop}\label{carrelations} The $*$-algebra generated by $\ee(V)$, where $\ee$ denotes the outer multiplication on the exterior algebra, yields a representation of the CAR algebra  by setting
$$\ee_k:=\ee(e_k)\text{ and }\ees_k:=(\ee_k)^*$$
where $\{e_k\}$ is an orthonormal basis for $V$.

Further, we have the relation
$$\sum_k\ees_k\ee_k=\sum_p (n-p)\id_{\Lambda^p_\mC (V)}.$$
\end{Prop}
\begin{proof}
That the $\ee_k$ and $\ees_k$ fulfil the relations of the CAR algebra is straightforward to show. Further, if $x=e_{k_1}\wedge\ldots\wedge e_{k_p}$, then
\begin{align*}\sum_k\ees_k\ee_k(x)=&\sum_k \ees_k(e_k\wedge e_{k_1}\wedge\ldots\wedge e_{k_p})\\
=&nx-\sum_{q=1}^p(-1)^{q-1}\delta_k^{k_q}e_k\wedge e_{k_1}\wedge\ldots\wedge\hat e_{k_q}\wedge\ldots\wedge e_{k_p}\\
=&nx-\sum_{q=1}^pe_{k_1}\wedge\ldots\wedge e_{k_p}\\
=&(n-p)x.
\end{align*}
\end{proof}
\begin{Prop}\label{gammarel} Let $\{e_k\}$ be a basis of $V$. Then the CAR algebra associated to this basis yields anticommuting representations of the Clifford algebras $\Cliff(V,-\Scp)$ and $\Cliff(V,\Scp)$ by setting\index{$\hat\gamma,\;\check\gamma$}
$$\hat\gamma_k:=\ee_k-\ees_k\text{ or } \check\gamma_k:=\ee_k+\ees_k.$$
Furthermore, we have the relations
$$\{\hat\gamma_k,\check\gamma_l\}=0\;\text{ and }\;\delta^{kl}\hat\gamma_k\check\gamma_l=n-2\sum_k\ees_k\ee_k.$$
\end{Prop}
\begin{proof} Follows from a direct calculation.
\end{proof}
This yields
\begin{Cor}\label{cliffrep} There are two anticommuting representations $c_+$ and $c_-$ of $\Cliffc(V)$ on $\Lambda_\mC^* V$ induced by
$$c_+(v) :=\ee(v)+\ees(v)\text{ and }c_-(v):= i(\ee(v)-\ees(v)).$$
\end{Cor}

\subsection{Bott periodicity}

For details concerning $KK$-theory, we refer to \cite{KaspOp}; the unbounded picture was developed in \cite{BaajJulg}.

Let $V$ be a euclidean space. We set $\hsp:=\Omega_{\mC}(V):=L^2(V)\hat\otimes\Lambda_\mC^*V$ and $B:=\fn_0(V)\hat\otimes\Lambda_\mC^*V$. Here we view $\hsp$ as graded into forms of even and odd degree, and $B$ carries the grading coming from the canonical grading on the Clifford algebra. 
Note that $c_+(v)=\ee(v)+\ee(v)^*$ corresponds to multiplication on the left under the identification $\Lambda_\mC^*V$ with $\mC_n$. We denote $c_+$ and $c_-$ the extensions of $c_+$ and $c_-$ to representations of  $A:=\fn_0(V)\hat\otimes\Cliffc (V)$ on  $\hsp$. We denote $f:V\to \Cliffc(V)$ the inclusion of $V$ into its complex Clifford algebra.
\begin{Def} We define the unbounded regular operators $D_1^0$ in $\hsp$  and $D_2$ in $B$ by
$$\Dom(D_1^0)=\fn^\infty_c(V)\hat\otimes \Lambda_\mC^*V\text{ and }\Dom(D_2):=\{\omega\in B|f\omega\in B\}$$
and set
$$D_1^0:=d+d^*=\hat\gamma^i\partial_i\text{ and }D_2:=M_f$$
where $M_f$ denotes the operator of multiplication by $f$. Then $D_1^0$ is essentially self adjoint, and we denote by $D_1$ its closure.\index{$D_1,\; D_2$}
\end{Def}

\begin{Lem} $(\hsp,c_+,D_1)$ is an unbounded Kasparov $(\fn_0(V)\hat\otimes \Cliffc (V),\mC)$-module.
\end{Lem}
\begin{proof} Let $\mu\in\fn^\infty(V)\hat\otimes \Cliffc (V)$. If $\mu$ is just a smooth function $h$, $[D_1
,c_+(h)]=\hat\gamma^i\partial_i(h)$, and is therefore bounded if $h$ has bounded derivatives. \ifthenelse{\boolean{notes}}{This follows easily as switching to the Clifford algebras and calculating $[d+d^*,f]\omega\hat= [D,f]\omega=D(f\omega)-fD\omega=D(f)\omega+fD\omega-fD\omega$.}{}

We may thus assume $\mu$ is not a scalar function, and write it as a product $\mu_1\cdots\mu_k$, where the $\mu_i$ are forms of degree one. Hence 
$$[D_1,c_l(\mu)]=\sum_{i=1}^k(-1)^{i-1}c_l(\mu_1\cdots\mu_{i-1})[D_1,c_l(\mu_i)]c_l(\mu_{i+1}\cdots\mu_{k}),$$
and we may therefore reduce to the case where $\mu$ is of the form $f\cdot e$ for a smooth function $f$. We set $E:=\linspan(e)^\perp$ and decompose $\hsp=\Omega_{\mC}^*(E)\oplus \ee(e)(\Omega_{\mC}^*(E))$. We may further assume that $e=e_1$, where $e_1$ is the first basis vector of $\mR^n$, as the definition of $d+d^*$ is independent of the choice of a basis. With respect to this decomposition, $(d+d^*)$ and $c_l(\mu)$ decompose as
\begin{align*}c_l(he_1)=c_l(h)c_l(e_1)=\mb h&\\&h\me\mb&1\\1&\me=\mb & h\\h&\me\end{align*}
and
\begin{align*}\,\mb (d_E+d_E^*)&-\partial_1\\\partial_1&-(d_E+d_E^*)&\me,
\end{align*}
where for a real vector space $V$, $d_V$ denotes the complexified exterior derivative on $\Omega_{\mC}^*(V)$. From this it is easily seen that the commutator is bounded.

For $f$  a smooth function of compact support, $f(1+D_1^2)^{-1}f^*$ is compact, being a pseudodifferential operator of negative order on a compact manifold. Thence  $c_+(\mu)(1+D_1^2)^{-1/2}$ is compact for all $\mu\in C_c^\infty(V)\hat\otimes \Cliffc  (V)$, and the result extends to $C_0(V)\hat\otimes \Cliffc (V)$ by continuity.
\end{proof}
\note{\begin{Rem} For the smooth scalar valued function $h(x):=\frac{1}{x} \sin(x^3)\in \fn_0(\mR)\hat\otimes\mC_1$, the commutator $[D_1,h]$ is not a bounded function. Hence the condition that $[a,D]$ be bounded only on a dense subset in the definition of an unbounded module.
\end{Rem}}{}
\begin{Def} We set \index{$x_n,\; y_n$}
\[x_n:=[(\hsp,c_+,D_1)]\in KK(S^n\hat\otimes\mC_n,\mC) \text{ and }y_n:=[(B,1,D_2)]\in KK(\mC,S^n\hat\otimes\mC_n).\]\end{Def}

The elements $x_n$ and $y_n$ are multiplicative (up to sign, which we neglect):
$$x_n\cup x_m=x_{n+m}\text{ and }y_n\cup y_m=y_{n+m}.$$
Hence they can be built up by iterated cup products of the one dimensional $x_1$ and $y_1$: More explicitly, on the Hilbert space $\hsp_1:=L^2(\mR)\oplus L^2(\mR)$, consider the unbounded, densely defined operator 
$$\mb &-\partial\\\partial&\me:\; (g,h)\mapsto (-\partial h,\partial g).$$
This corresponds exactly to the operator $d+d^*$ under the isomorphism 
$$\phi:L^2(\mR)\oplus L^2(\mR)\to L^2(\mR)\hat\otimes\Lambda_\mC^*\mR, \; (g,h)\mapsto g+hdx.$$
If we identify $\fn_0(\mR)\hat\otimes\mC_1$ with $\fn_0(\mR)\oplus \fn_0(\mR)$ with the standard odd grading, and let the first summand act by diagonal matrices, the second by off-diagonal ones, we obtain again a Kasparov module. Denote the coordinates on a $k$'th copy $\hsp_k$ of $\hsp=L^2(\mR)$   by $x^k$, set 
$$S_k:=1\hat\otimes\cdots\hat\otimes\mb & -\partial_k\\\partial_k&\me\hat\otimes\cdots\hat\otimes 1,$$
where the only entry not equal to $1$ is in the $k$-th tensor-factor. Define 
$$S :=\sum_{k=1}^n S_k\in\mathbb{B}(\bigotimes\limits_{k=1}^n \hsp_k).$$
Under the isomorphism
$$\hat\bigotimes_{k=1}^n \hsp_k \to L^2(\mR^n)\hat\otimes \Lambda_\mC^*\mR^n,\; (g_1,h_1)\hat\otimes\cdots\hat\otimes (g_n,h_n)\mapsto (g_1+h_1 dx^1)\cdots(g_n+h_ndx^n)$$
this operator is easily seen to coincide with $d+d^*$. The Clifford algebras are well known to be multiplicative (for the graded tensor product).

Similarly, we may decompose $\fn_0(\mR^n)\hat\otimes \mC_{n}\approx \hat\bigotimes_{k=1}^n \fn_0(\mR)\hat\otimes \mC_{1}$, and hence get multiplicativity of the elements $y_n$ and a decomposition of the (unbounded) operator 
$$T:=\sum_{k=1}^n 1\hat\otimes\cdots\hat\otimes \mb & {u_k} \\ {u_k} & \me\hat\otimes\cdots \hat\otimes 1,$$
where the $u_k$ are the coordinate functions. Note the crucial difference in the sign conventions between the definition of $D_1=S$ and $D_2=T$.

\begin{Prop}\label{uglyrelation} The operators $D_1$ and $c_+(f)$ in $\hsp$ fulfil the following relations
$$[D_1,c_+(f)]=\sum_{p}(-n+2p)\id_{\Omega^p_\mC(V)}\text{ and }(D_1+c_+(f))^2=H+[D_1,c_+(f)],$$
where $H=-\Delta+x^2$.
\end{Prop}
\begin{proof}
Using the relations from Proposition \ref{gammarel} and Proposition \ref{carrelations}, we get
\begin{align*}
[D_1,c_+(f)]=&\hat\gamma^i\partial_i(x_j\check\gamma^j)+x_j\check\gamma^j\hat\gamma^i\partial_i\\
=&\delta_{i,j}\hat\gamma^i\check\gamma^j+x_j(\hat\gamma^i\check\gamma^j+\check\gamma^j\hat\gamma^i)\partial_i\\
=&\delta_{i,j}\hat\gamma^i\check\gamma^j\\
=&n-2\sum_i\ees_i\ee_i\\
=&\sum_{p}(-n+2p)\id_{\Omega^p_\mC(V)}
\end{align*}
\note{Also
\begin{align*} (D_1+c_+(f))^2=D_1^2+c_+(f)^2+[D_1,c_+(f)]=-\Delta+x^2+[D_1,c_+(f)]
\end{align*}}
\end{proof}
\note{
\begin{Prop} The class of $H-1$ is one and the operator $\bar D_1$ defined as the closure of $d+d^*+c_+(x)$ has summable resolvent.
\end{Prop}
\begin{proof}
This follows essentially from the theory of the quantum harmonic oscillator. Hence we define as usually the ladder operator $a:=\frac{1}{\sqrt{2}}(x+\partial)$ and the particle number operator $N:=a^*a$. Further the Hermite functions are defined in terms of the Hermite polynomials $H_n$ by $\psi_n:=C_ne^{-\frac{x^2}{2}} H_n$, where $C_n:=(n!2^n\sqrt{\pi})^{-1/2}$. Then we have the classical relation
$a^*\psi_k=\sqrt{k}\psi_{k-1}$, and using $[N,(a^*)^n]=n(a^*)^n$ one easily shows that $a\psi_k=\sqrt{k}\psi_{k-1}$.
\end{proof}}{}
We recall:
\begin{Lem}\label{Hermite} The Hermite functions $\psi_n$ are an orthonormal basis of $L^2(\mR)$ such that
$$(-\partial^2+x^2-(2n+1))\psi_n=\psi_n.$$
\end{Lem}
\begin{Prop}\label{newdirac} If we define the unbounded operator $\bar D$\index{$\bar D$} on $\hsp$  as the closure of the operator $(d+d^*+c_+(x))$ with domain the smooth complex forms of compact support, then $(\hsp,1,\bar D)$ is an unbounded cycle representing the product $y_n\cap x_n$.
\end{Prop}
\begin{proof}  As for every  $h$ of compact support and $v\in\mC_n$ the operator $h(\bar D-D_1)$ is bounded, $(\hsp,c_+,\bar D)$ is a spectral triple, being a bounded perturbation of $(\hsp,c_+,D_1)$. Further, $\bar D$ itself has compact resolvent by the above Lemma, hence $(\hsp,1,\bar D)$ is a spectral triple.

Recall that $B:=\fn_0(\mR^n)\hat\otimes\mC_n$, and denote by  $\theta:B\hat\otimes_{c_+}\hsp\to\hsp,\,\mu \overset{.}{\otimes}\omega\mapsto c_+(\mu)\omega$ the canonical isomorphism. Then $\theta^{-1}q(D_1)\theta$ is a $B$-connexion for $q(D_1)$, as 
$$\theta(\mu\overset{.}{\otimes}q(D_1)\omega-(-1)^{\partial \mu} \theta^{-1}q(D_1)\theta(\mu\overset{.}{\otimes}\omega))=[c_+(\mu),q(D_1)](\omega)$$
is compact. \note{This is essentially the proof (when we don't need to stabilize the module) that connections exist.}  Hence as $q(\bar D)$ is a compact perturbation of $q(D_1)$:
\begin{align*} &\theta(\mu\overset{.}{\otimes}q(D_1)\omega-(-1)^{\partial \mu} \theta^{-1}q(\bar D)\theta(\mu\overset{.}{\otimes}\omega))\\
&=[c_+(\mu),q(\bar D)]\omega-((-1)^{\partial\mu}(q(\bar D)-q(D_1))c_+(\mu))\omega
\end{align*}
is compact operators, showing that $q(\bar D)$ is a connection for $q(D)$.\note{The point is here: Even though the transported $D$ is a connection, we need to show that $\bar D$ is one. But only $a(q(D)-q(D'))$ is compact, hence the trick to transport via $\theta$ in order to be in the commutator, so one does multiply by some $a$.}{}
It remains only to check the positivity condition for $q(\bar D)$. The operator in $y_n$ corresponds, under $\theta$, to $C:=q(c_+(x))$, and we set  $T_t:=(1+t+\bar D^2)^{-1}\in\mathbb{K}(\hsp)$.  First of all, $\bar D[T^{1/4}_0,C]$ is compact, as
$$\bar D[T^{1/4}_0,C]=c\int_0^\infty \bar D T_t[C,\bar D^2]T_t\frac{dt}{t^{1/4}}$$
for a constant $c$; we have here applied the integral formula from \cite{UC}.\note{Here one needs to take care of the commutator with $\bar D^2$: $[C,\bar D^2]$! But $C=\frac{c_+(x)}{\sqrt{1+x^2}}$, hence it preserves the smooth functions with compact support (it is a pseudodifferential operator). Further, one could write this as $$\int \bar DT_t(\bar D[C,\bar D]+[C,\bar D]\bar D)T_t$$
the which makes sense (even though I have not found a proof of the equality).}{} Further $$T_0^{1/4}[C,c_+(x)]T_0^{1/4}=T_0^{1/4}\frac{x^2}{\sqrt{1+x^2}}T_0^{1/4}\geq 0.$$
 This shows that all the summands in the decomposition of the commutator are either compact or positive.
\end{proof}

\begin{Th}\label{Bottperiodicity} $x_n$ and $y_n$ yield $KK$-equivalences between $S^n\hat\otimes\mC_n$ and $\mC$, more precisely, $x_n\cap y_n=1_{S^n\hat\otimes \mC_n}$ and $y_n\cap x_n=1_{\mC}$.
\end{Th}
\begin{proof} We will calculate $y_n\cap x_n=1_{\mC}$ and deduce $x_n\cap y_n=1_{S^n\hat\otimes\mC_n}$ from this by the Atiyah rotation trick (cf. \cite{MR0228000}). As the elements $x_n$ and $y_n$ are multiplicative and the Kasparov product associative, it suffices to prove the case $n=1$.
 
By Proposition \ref{newdirac}, all we have to do is calculate the Fredholm index of $\bar D^{0}:=\partial+x$ on $L^2(\mR)$. \note{ This follows because the isomorphism $KK(\mC,\mC)\to \mZ$ is given by sending the operator $\mb &T^*\\T&\me$ to the Fredholm index of $T$, as can be seen from, for example \cite{BlackK}, Example 17.3.4, or even from Proposition \ref{indeximports} in the $C^*$-setting below; in order to pass from the index of the operator $q(D)$ to the index of $D$, it suffices to observe that $(1+D^*D)^{-1/2}$ is an operator onto $\Dom(D)$ that is injectif.}{}Now $\Ker(D^*D)=\Ker(D)$ for any densely defined unbounded operator. 
Hence by Lemma \ref{Hermite}
$$\dim \Ker(\bar D^{0})=\dim\Ker(-\partial^2+x^2-1)=1,$$
and $\dim\Coker(\bar D^{0})=0$, thus $\ind (\bar D^{(0)})=1$. This shows that $y_n\cap x_n=1_{\mC}$.

Now comes the rotation trick due to Atiyah, which deduces $x\cap y$ from $y\cap x$. Denote by $\sigma$ the flip of $(\fn_0(\mR)\hat\otimes \mC_1)\hat\otimes(\fn_0(\mR)\hat\otimes\mC_1)$.

Since $SO(n)$ acts on $\fn_0(\mR^n)\otimes\mC_n$, the complex unit $i$ induces another isomorphism $i_*$ of $(\fn_0(\mR)\hat\otimes \mC_1)\hat\otimes(\fn_0(\mR)\hat\otimes\mC_1)$.

The $*$-automorphisms $\sigma$ and $i_*$ are determined by their action on a subset which generates (as an algebra) a dense $*$-subalgebra; hence, denoting the basis vectors of $\mR^2$ by $e_1$ and $e_2$,  they are given by
\begin{align*}
\sigma:\,\,&f(x)  e_1\mapsto f(y)e_2\\
& g(y)e_2\mapsto g(x)e_1\\
i_*\,\,:&f(x)e_1\mapsto f(y)e_2\\
		&g(y)e_2\mapsto -g(-x)e_1.
\end{align*}
Therefore $i_*\circ (1\otimes \rho)=\sigma$, where $\rho$ denotes the automorphism of $\fn_0(\mR)\hat\otimes\mC_1$ induced by $-\id_\mR$.

Homotoping $i$ to $\id_\mC$, i.e., using that $SO(2)=S^1$ is connected, we see that $\sigma$ is homotopic to $1\otimes \rho$. Set  $D:=\fn_0(\mR^2)\hat\otimes\mC_2$. Then using that the cup-product   coincides with the cap-product over $\mC$,  we obtain
\begin{align*}x\cap y=y\cup x=&\tau^r_D(y)\cap\tau^l_D(x)\\
=&\tau^r_D(y)\cap \sigma_*(\tau^r_D(x))\\
=&\tau^r_D(y)\cap (1\otimes \rho)_*(\tau^r_D(x))\\
=&(1\otimes \rho)_*(\tau^r_D(y\cap x))\\
=&(1\otimes\rho)_*.
\end{align*}
Thence $x$ has $y$ as a left inverse by the first part of the proof, and it is right invertible also - consequently $y$ must be this right inverse as well.
\end{proof}

\begin{Rem} We have used that if $D$ is a densely defined unbounded operator with adjoint $D^*$, then $\Ker(D^*)=\Im (D)^\perp\cap \mc{D}(D^*)$, and therefore 
\begin{align*}&\Ker(D^*D)=\Ker(D)\cup(\Ker(D^*)\cap \Im(D))\\
=&\Ker(D)\cup(\Im(D)^\perp\cap \mc{D}(D^*)\cap \Im(D))=\Ker(D).
\end{align*}
\end{Rem}

\subsection{Some remarks concerning Bott Periodicity}
One may understand the form of the Dirac and Dual-Dirac a bit more conceptually as follows. For this, refine the definitions of section \ref{prepcliff} as follows: If $V$ is again a vector space with a positive definite scalar product, then denote now by $\ee_l(v)$ the exterior multiplication on the left by $v\in V$ on $\Lambda_\mC^*(V)$, and the exterior multiplication on the right by $\ee_r(v)$. We then also have their adjoints, accordingly denoted $\ee_l^*$ and $\ee_r^*$. 

By Corollary \ref{cliffrep}, we already have anticommuting representations of $\Cliffc(V)$ on $\Lambda^*_\mC(V)$ on the left by setting $c_+^l=\ee_l+\ee_l^*$ and $c_-^l=i(\ee_l-\ee_l^*)$ of $\Cliffc(V)$.

Further, it is easily seen that under the identification of vector spaces $\Lambda_\mC^*V\approx \Cliffc(V)$ the natural action of $\Cliffc(V)$ on itself by right multiplication $c_+^r$ corresponds to $\ee_r+\ee^*_r$. And there is yet a fourth representation $c_-^r$ of $\Cliffc(V)$ on the right defined by $i(\ee_r-\ee_r^*)$. 

The representations $c_+^r$ and $c_-^r$ again anticommute, and $c_+^r$ commutes with $c_+^l$ because they correspond to right and left multiplication on $\Cliffc(V)$. The definition of the element $y_n$ from above is thus actually based on the use of $c_+^l$ for the operator and $c_+^r$ for the right action. With this at hand, we may define another version of the Dirac and Bott elements. We denote by $f_-$ the function $V\to\mB(\Lambda_\mC^*V)$ given by $v\mapsto c^l_-(v)$, and by $\phi:\fn_0(V)\to\bo$ the action by multiplication.
\begin{PD} If we define the unbounded operators $D_1'$ and $D_2'$ by 
$$\Dom(D_1'):=\fn_c^\infty(V)\hat\otimes \Lambda_\mC^*V\text{ and }\Dom(D_2'):=\{\omega\in B|f_-\omega\in B\}$$ 
and set 
$$D_1':=i\check\gamma^k\partial_k\text{ and } D_2':=M_{f_-}$$
then we obtain Kasparov $(\fn_0(V),\Cliffc(V))$- and $(\Cliffc(V),\fn_0(V))$-modules $(\hsp,\phi,D_1')$ and $(B,c^l_+,D_2')$.
\end{PD}
\begin{Rem} One easily checks that, for example $c_+^l(v)=(\ee_r-\ee_r^*)(v)\circ\gamma$, where $\gamma$ denotes again the grading operator; hence all these actions are all distinct. It is also straightforward to see that the element $D_1$ from the first definition, for example, corresponds to the left handed Dirac operator for the scalar product $-\Scp$.
\end{Rem}

\subsection{The Thom isomorphism in $KK$}
\ifthenelse{\boolean{notes}}{Bundles of finitely generated $A$-modules, where $A$ is a $C^*$-algebra, and pseudodifferential operators on sections of such bundles have been introduced by Mishchenko and Fomenko in \cite{MR548506}. We will have to deal with a similar type of bundles for the proof of the Thom-isomorphism. The difference is that we have to produce modules in the smooth category and our fibres are of more general type.}{}
\ifthenelse{\boolean{notes}}{\marginpar{use canonic product on finitely generated $A$-modules to get inner product?!}}{}

In this chapter, we discuss how principal $G$-bundles induce a functor from the category  $KK_G$ to $KK$ (more precisely, $\mathcal{R}KK$).

In the whole chapter, we assume that all our fibre bundles $p:E\twoheadrightarrow X$ are locally trivial; the projection is then open\ifthenelse{\boolean{notes}}{: If $e\in E$, $V\subseteq E$ a neighbourhood of $e$ and $\pi(e):=x$, then we may choose an open set $x\in U\subseteq X$ over which $E$ is trivialized by a bundle morphism $\rho:E|_U\to U\times F$, and $\pi_1(\rho(E|_U\cap V))$ is a neighbourhood of $x$.}{} (for principal bundles $p$ is always open.) Also, if the structure group and base are (locally) compact, the corresponding principal bundle is (locally) compact. In fact, it is not hard to show the stronger statement that the projection is closed when the fibre is compact.
\ifthenelse{\boolean{notes}}{\begin{Lem}
Let $p:E\twoheadrightarrow X$ be a locally trivial fibre bundle. If  the fibre is compact, then the projection is closed.
\end{Lem}
\begin{proof}
Let $A\subseteq E$ be closed, set $A':=p(A)$ and choose $x\in X\setminus A'$. Let $x\in U\subseteq X$ be an open trivializing set with bundle chart $\phi:E|_U\to U\times F$. If $A'\cap U=\emptyset$, we are done. Otherwise $\bar A:=\phi(A\cap E|_U)$ is closed in $U\times F$. For every point $(x,f)$ over $x$ choose an open set $V_f\times W_f$  in the complement of $\bar A$ containing $(x,f)$, and choose a finite subset $f_i$ such that $W_{f_i}$ covers $F$. Set $V:=\cap V_{f_i}$; then $V\times F$ is open in $U\times F$ and as $\pi_1$ is open, $V$ is a neighbourhood of $x$ in $A'\setminus U$, and as $U$ is open, also in $X$. Hence $A'$ is closed.
\end{proof}
\begin{Prop}[Twisted Tychonoff Theorem] Let $p:E\twoheadrightarrow X$ be a locally trivial fibre bundle. Then if the fibre and base are compact, the total space $E$ is compact. 
\end{Prop}
\begin{proof} By the above, the projection $p$ is a closed map. Then we use the well known fact that a  continuous closed map of topological spaces such that the preimage of every point is compact is proper. Hence $p^{-1}(X)=E$ is compact.\end{proof}}{}

We assume in this chapter that all spaces are locally compact and  all groups are compact.

\com{Should get correspondence $\Gamma(X,E)=G(P,E)$ ($G\leftrightarrow \Gamma$, $X\leftrightarrow P$), all regularity conditions should stay fixed (ex:$\Gamma_c^n\leftrightarrow G_c^n$)}{}

For any locally compact group $G$ and right $G$-spaces $F$, $F'$ we denote by $\fn_G(F,F')$ the continuous equivariant maps from $F$ to $F'$ that vanish at infinity. For a bundle $E\twoheadrightarrow X$ over a topological space  we denote  by $\Gamma(E)$ the continuous sections vanishing at infinity.

When speaking of principal bundles, we will always assume that the action of the appropriate group on the principal bundle is on the right. As usual, a topological space with continuous action of a topological group $G$ will be called a $G$-space.
\subsection{Fibrations}\label{fibrations}
Let $X$ be a topological space, $P\twoheadrightarrow X$ a principal $G$-bundle for a group $G$ and $Y$ a left $G$-space. Then 
$$\Gamma(P\times_G Y)\cong \fn_G(P,Y)$$
where $Y$ is viewed as a right $G$-space. In particular
\begin{equation*}\Gamma(P\times_G \fn (Y))\cong \fn (P\times_G Y)
\end{equation*}
where $G$ acts on $\fn (Y)$ by $(g.f)(y):=f(g^{-1}y)$ for all $g\in G$, $y\in Y$ and $f\in \fn (Y)$. In particular, $\Gamma (P\times_G \fn (\mR^n))\cong \fn(E)$, where $E\twoheadrightarrow X$ is a vector bundle associated to the principal $G$-bundle $P\twoheadrightarrow X$ with fibre a $G$-space $\mR^n$.

In general, given a topological group $G$, a topological left $G$-space $F$ and topological principal $G$-bundle $P\twoheadrightarrow X$ over a topological space $X$, we set $F_P:=\Gamma(P\times_G F)$. Then $\Gamma(P\times_G\cdot)=\cdot_P$ is actually a functor from the category of left $G$-spaces and $G$-maps to the category of "$\fn(X)$"-objects  in the target category; if $\phi:A\to B$ is a $*$-homomorphism, we denote by $\phi_P$ the map $A_P\to B_P$.

A $B$-Hilbert module $E$ is called a $G$-Hilbert module, if $B$ is a $G$-algebra and the group $G$ acts on $E$ continuously through $B$-linear operators such that for all $g\in G$, $\xi,\eta\in E$, $b\in B$:
$$g(\xi b)=(g\xi)(gb)\; \text{ and }\;\langle g\xi|g\eta\rangle=g\langle \xi|\eta\rangle.$$
$E_P$ is then a Hilbert $B_P$-module by setting for all $\tilde \xi,\tilde\eta\in \fn_G(P,E)\approx \Gamma(P\times_G E)$
$$\langle\tilde\xi|\tilde\eta\rangle(p):=\langle\tilde\xi(p)|\tilde\eta(p)\rangle,$$
which is clearly an element of $\fn_G(P,B)$. 

\begin{Def} A $\fn(X)$-algebra, where $X$ is a locally compact space, is a $C^*$-algebra $A$ with a homomorphism $\fn(X)\to Z(\mc{M}(A))$ into the center of the multiplier algebra of $A$ that is nongenerate, i.e. $\fn(X) A=A$.
\end{Def}
Note that $\fn(X)A=A$ is equivalent to $\fn(X)A$ is dense in $A$, by Proposition 1.8 in \cite{MR1395009}. 
Every Hilbert $A$-module $E$ over a $\fn(X)$-algebra inherits a $\fn(X)$-action through bounded operators because  $E$ can be viewed as an $\mathcal{M}(A)$-module; every $T\in\mathbb{B}(E)$ is then automatically $C(X)$-linear.
\begin{Prop}\label{fibrewiseops}
Let $E$ be a $G$-Hilbert module over a $G$-$C^*$-algebra $B$, $P\twoheadrightarrow X$ a principal $G$-bundle. Then 
$$\mathbb{B}_{B_P}(E_P)\cong \Gamma_b(P\times_G(\mathbb{B}_B(E),str))$$
where $\mathbb{B}_B(E)$ carries the strict operator topology and the sections are bounded in norm. Further
$$\mathbb{K}_{B_P}(E_P)\cong \Gamma_b(P\times_G( \mathbb{K}_B(E),||\cdot||_{op})).$$
The spectrum of an operator is the closure of the union of the spectra in the fibres.
\end{Prop}
\begin{proof} This follows easily by using the description of sections as equivariant maps and local trivialisations.
\end{proof}

We suppose for the rest that all our spaces are second countable.

\begin{PD} Let $P\twoheadrightarrow X$ be a principal $G$-bundle over $X$. For every Kasparov module $x=(E,\phi,F)$ we define a new Kasparov module (using Proposition \ref{fibrewiseops})\index{$x_P$}
$$x_P:=(E_P,\phi_P,F_P).$$ 
\end{PD}
\begin{proof} 
Obvious by \ref{fibrewiseops}.
\end{proof}

The definition of $\mathcal{R}KK_X(A,B):=\mathcal{R}KK(X;A,B)$ is given in \cite{KasparovNovikov}. 
\begin{Prop}\label{fibprod} $\cdot_P$ yields a functor from $KK_G$ to $KK$, more precisely with values in $\mathcal{R}KK_X$. In particular: $x_P\cap y_P=(x\cap y)_P$.
\end{Prop}
\begin{proof} The only nontrivial part remaining is compatibility with the product. Using an equivariant product for two given modules, it is easily seen that the fibration of the product is a product of the fibrations of the modules using Proposition \ref{fibrewiseops}. 
\end{proof}

We give a proof of a version of the Thom isomorphism that was stated without proof in \cite{KasparovNovikov}. It says that the two  algebras one can naturally	 associate to a bundle -- the functions on the total space and sections of the cliffordified bundle -- are naturally the same in $KK$-theory.
\begin{Th} If $E\twoheadrightarrow X$ is a vector bundle, then the associated algebras $\Gamma(\mC l(E))$ and $\fn(E)$ are isomorphic in $KK$, i.e., admit a $KK$-equivalence.

If $E$ is a $spin^c$-bundle, then $\Gamma(E)$ and $\fn(X)$ are $KK$-equivalent (through an element of $KK^1$in odd dimension).
\end{Th}
\begin{proof}
We choose a hermitian metric on $E$ and write $E$ as the associated bundle to an $O(n)$-principal bundle, $E=P\times_{O(n)}\mR^n$. The complex Clifford bundle does not depend on the choice of hermitian metric and decomposes as $\mC l(E)=P\times_{O(n)}\mC_n$. 

Let $[x_n]\in KK_{O(n)}(\fn(\mR^n)\hat\otimes \mC_n,\mC)$ be the class of the Bott element and $[y_n]$ its inverse. Then equipping a second copy of $\mC_n$ with the same $O(n)$-action, we can define a $(\fn(\mR^n)\hat\otimes\mC_{2n},\mC_n)$-module $x:=\tau^r_{\mC_n}(x_n)$. 
This yields modules 
$$x_P:=(\tau^r_{\mC_n}(x_n))_P\text{ and }y_P:=(\tau^r_{\mC_n}(y_n))_P$$
that are inverse to each other by \ref{fibprod}, and hence 
$$\Gamma(\mC l(E))\sim\Gamma((P\times_G \fn(\mR^n))\hat\otimes_{\fn(X)}\Gamma(\mC l(E\oplus E))$$ 
are $KK$-equivalent. The bundle  $\mC l(E\oplus E)$ has a canonical complex structure, is thus $spin^c$, and hence Morita equivalent to $\fn(X)$ (by \cite{MR860349}, 2.11 Theorem); Morita invariance of $KK$ now proves the theorem.
\end{proof}

\bibliographystyle{alpha}
\bibliography{../../Fullbib}
\end{document}